\documentclass[11pt]{amsart}

 \usepackage{amsopn}
 \usepackage{amsmath,amsthm,amssymb}
 \usepackage[hypertex]{hyperref}

 \newcommand{\nc}{\newcommand}

\nc{\bb}{\mathfrak{b} }
 \nc{\cc}{\mathfrak{c} }  \nc{\dd}{\mathfrak{d} } 
 \newcommand\ee{{\mathfrak e}}   \nc{\ggo}{\mathfrak{g} }
 \nc{\hh}{\mathfrak{h} }  \nc{\ii}{\mathfrak{i} }
 \nc{\jj}{\mathfrak{j} }  \nc{\kk}{\mathfrak{k} }
\nc{\mm}{\mathfrak{m} }   \nc{\nn}{\mathfrak{n} }
\nc{\pp}{\mathfrak{p} }   
\nc{\rr}{\mathfrak{r} } \nc{\sg}{\mathfrak{s} }
 \nc{\sso}{\mathfrak{so} }  \nc{\spg}{\mathfrak{sp} }
 \nc{\ssu}{\mathfrak{su} }  \nc{\ssl}{\mathfrak{sl} }
 \nc{\tog}{\mathfrak{t} }  \nc{\uu}{\mathfrak{u} }
 \nc{\vv}{\mathfrak{v} } \nc{\ww}{\mathfrak{w} }
 \nc{\zz}{\mathfrak{z} }

\nc{\CC}{{\mathbb C}}
 \nc{\DD}{{\mathbb D}}
\nc{\FF}{{\mathbb F}}
\nc{\GG}{{\mathbb G}}  
\nc{\HH}{{\mathbb H}}
\nc{\II}{{\mathbb I}}
\nc{\JJ}{{\mathbb J}}
\nc{\KK}{{\mathbb K}}
\nc{\NN}{{\mathbb N}}

\nc{\RR}{{\mathbb R}}  
 \nc{\ZZ}{{\mathbb Z}}

\nc{\ggob}{\overline{\mathfrak{g}}} 
 
\nc{\glg}{\mathfrak{gl} }
  
\nc{\pca}{\mathcal{P}} \nc{\nca}{\mathcal{N}}
 
 \nc{\vp}{\varphi} \nc{\ddt}{\frac{{\rm d}}{{\rm d}t}}
 \nc{\la}{\langle} \nc{\ra}{\rangle}
 \nc{\brg}{[\,,\,]_{\ggo}}
 \nc{\brv}{[\,,\,]_{\vv}}
 
 \nc{\SO}{{\sf SO}} \nc{\Spe}{{\sf Sp}} \nc{\Sl}{{\sf Sl}}
 \nc{\SU}{{\sf SU}} \nc{\Or}{{\sf O}} \nc{\U}{{\sf U}}
 \nc{\Gl}{{\sf Gl}} \nc{\Se}{{\sf S}} \nc{\Cl}{{\sf Cl}}
 \nc{\Spin}{{\sf Spin}} \nc{\Pin}{{\sf Pin}}

 \nc{\ad}{\operatorname{ad}} \nc{\Ad}{\operatorname{Ad}}
 \nc{\coad}{\operatorname{coad}}
 \nc{\rank}{\operatorname{rank}} \nc{\Irr}{\operatorname{Irr}}
 \nc{\End}{\operatorname{End}} \nc{\Aut}{\operatorname{Aut}}
 \nc{\Inn}{\operatorname{Inn}} \nc{\Der}{\operatorname{Der}}
 \nc{\Ker}{\operatorname{Ker}} \nc{\Iso}{\operatorname{I}}
 \nc{\Le}{\operatorname{L}} \nc{\tr}{\operatorname{tr}}
 \nc{\dif}{\operatorname{d}} \nc{\sen}{\operatorname{sen}}
 \nc{\modu}{\operatorname{mod}} \nc{\Ric}{\operatorname{R}}
 \nc{\Sym}{\operatorname{Sym}} \nc{\sca}{\operatorname{sc}}
 \nc{\scalar}{{\sf s}} \nc{\grad}{\operatorname{grad}}
 \nc{\ricci}{\operatorname{r}} \nc{\riccin}{\operatorname{Ric}}
 \nc{\Lie}{\operatorname{L}} \nc{\ct}{\operatorname{T}}
 
 \theoremstyle{plain}
 \newtheorem{thm}{Theorem}[section]
 \newtheorem{prop}[thm]{Proposition}
 \newtheorem{cor}[thm]{Corollary}
 \newtheorem{lem}[thm]{Lemma}
 
 \theoremstyle{definition}
 \newtheorem{defn}[thm]{Definition}
 
 \theoremstyle{remark}
 \newtheorem*{rem}{Remark}
 
 \newtheorem{exa}[thm]{Example}

 \newcommand{\ri}{{\rm (i)}}
 \newcommand{\rii}{{\rm (ii)}}
 \newcommand{\riii}{{\rm (iii)}}
 
 \newcommand{\rv}{{\rm (v)}}

\begin{document}
\title[Naturally reductive pseudo Riemannian  2-step nilpotent Lie groups]
{Naturally reductive pseudo Riemannian \\ 2-step nilpotent Lie groups}

\author{Gabriela P. Ovando $^1$}
\address{G. P. Ovando: CONICET and ECEN-FCEIA, Universidad Nacional de Rosario 
\\Pellegrini 250, 2000 Rosario, Santa Fe, Argentina}
\footnote{ Currently: Abteilung f\"ur Reine Mathematik, Albert-Ludwigs
Universit\"at 
Freiburg, Eckerstr.1, 79104 Freiburg, Germany.}
\email{gabriela@fceia.unr.edu.ar}

%\date{\today}

\begin{abstract} This paper  deals with naturally reductive pseudo 
Riemannian  2-step nilpotent Lie groups $(N, \la \,,\,\ra_N)$. In the cases under consideration they are related to bi-invariant metrics. On the one hand, whenever  $\la \,,\, \ra_N$  restricts to a metric in the center it is proved that the simply connected Lie group $N$ arises from a Lie algebra  $\ggo$ and a representation of it. The Lie algebra 
$\ggo$ carries an ad-invariant metric and its corresponding Lie group acts as a group of  isometries of $(N, \la \,,\,\ra)$ fixing the identity element.  On the other hand,  a bi-invariant metric $\la\,,\,\ra$ on $N$ provides another family of examples of naturally reductive spaces, namely those of the form $(N/\Gamma, \la\,,\,\ra)$ being $\Gamma\subset N$ a  lattice, which are also investigated. \end{abstract}

\thanks{{\it (2010) Mathematics Subject Classification}:  53C50 22E25 53B30
53C30. }

\maketitle

\section{Introduction}

%A homogeneous pseudo Riemannian manifold $(M,g)$ is  provided with  a
%transitive group of isometries. Such a manifold can be identified with a 
%quotient $(G/H, g)$ where $G$ is a Lie group and $H<G$ is the isotropy 
%subgroup at a fixed point $o\in M$. 

The 2-step nilpotent groups are nonabelian and from the algebraic point of
view as close as possible to be Abelian and  they evidence a rich geometry when
equipped with a metric tensor. While they have been extensively investigated
  in the Riemannian situation, in the case of indefinite metrics, there 
  are significant advances as showed in \cite{Bo, C-P1, C-P2, Ge, 
J-P-L,J-P,J-P-P,Pa} but
there are still several open problems. A first obstacle appears when
trying to traduce the left invariant metric to the Lie algebra level. So far
all attempts in this direction take as starting point  the Riemannian 
model. Among these pseudo Riemannian spaces, the {\em naturally reductive} ones  are endowed with nice simple algebraic and geometric
properties. Examples of them are provided by 2-step nilpotent Lie groups carrying a bi-invariant metric.  
 
  Important  
studies  concerning the structure of a  naturally reductive Riemannian Lie group $G$ when $G$ is compact and simple or when $G$ is non compact and 
semisimple were given  by D'Atri-Ziller \cite{DA-Z} and Gordon \cite{Go} respectively. Gordon showed that every naturally reductive Riemannian manifold may be realized as a
homogeneous space $G/H$ with Lie group $G$ of the form $G=G_{nc}G_cN$ where $G_{nc}$ is
a non compact semisimple normal subgroup, $G_c$ is compact semisimple and $N$ is the
nilradical of $G$.  Furthermore $N\cap H=\{0\}$ and the induced metrics on each of
$G_{nc}/(G_{nc}\cap H)$, $G_c/(G_c \cap H)$ and $N (=N/(N\cap H))$ are naturally
reductive so that the study of naturally reductive metrics is partially reduced to
the cases in which $G$ is semisimple either of compact or non compact type or $G$ is
nilpotent. In the last case Gordon proved that $G$ must be at most 2-step nilpotent. 
 Lauret \cite{La} exploited this result to afford a classification of naturally reductive Riemannian
connected simply connected nilmanifolds. According to Wilson \cite{Wi} such a
manifold  can be realized as a 2-step nilpotent Lie group equipped with a left
invariant metric. 

Later Tricerri and Vanhecke \cite{T-V2} proved that a Riemannian manifold is 
a naturally reductive homogeneous space if and only if there exists a
homogeneous structure $T$ satisfying $T_x x=0$ for all tangent vectors $x$, offering
in this way an infinitesimal description of these reductive manifolds.
The notion of {\em homogeneous structure} was introduced by 
Ambrose and Singer \cite{A-S} to characterize connected simply connected and 
complete homogeneous Riemannian manifolds. In the Riemannian case every homogeneous 
manifold is complete and reductive.  
More recently Gadea and Oubi\~na \cite{G-O1} proved that a
connected simply connected and complete pseudo Riemannian manifold admits a
pseudo-Riemannian structure if and only if it is reductive homogeneous. While  Tricerri and Vanhecke \cite{T-V1} achieved the 
classification of  homogeneous  Riemannian structures, 
in the pseudo Riemannian case, a complete classification is still a pending
item.  

However Calvaruso and
Marinosci \cite{Ca,C-M1, C-M2} studied homogeneous structures  in dimension three, obtaining with their results the  naturally reductive  Lie groups with a left
invariant Lorent\-zian metric. In particular the  Heisenberg Lie group admits two
naturally reductive left invariant Lorent\-zian metrics (and for which the center 
 is non degenerate). 
 
In this paper we provide  constructions for  naturally reductive
pseudo Riemannian 2-step nilpotent Lie groups. By following a similar  approach to that one of  Gordon, one gets  necessary and sufficient conditions to have naturally reductive pseudo Riemannian 2-step nilpotent Lie groups with non degenerate center -Theorem (\ref{t1})-. This enables to attach this kind of Lie groups to
Lie algebras endowed with an ad-invariant metric and to certain kind of 
representations of them:  

\vskip 2pt

{\bf Theorem \ref{t2}} {\em Let $\ggo$ denote a Lie algebra carrying  an 
ad-invariant metric $\la\,,\,\ra_{\ggo}$ and let $(\pi, \vv)$ be a real 
faithful representation of $\ggo$ without trivial subrepresentations and 
 such that the metric on $\vv$, $\la\,,\,\ra_{\vv}$ is $\pi(\ggo)$-invariant. Let $\nn$
denote the  Lie algebra
 $\nn=\ggo \oplus \vv$ whose the Lie bracket is given  by
 $$\begin{array}{rcl}
  [\ggo,\ggo]_{\nn}=[\ggo, \vv]_{\nn} =0 & [\vv, \vv] \subseteq \ggo \\ \\
\la [u,v], x\ra_{\ggo} = \la \pi(x) u, v\ra_{\vv}& \mbox{ for all  } x\in \ggo,
\, \forall u, v\in
\vv, 
\end{array}
$$
equip $\nn$ with the metric 
 $\la\,,\,\ra$ 
 $$ \la\,,\,\ra_{\ggo\times \ggo}= \la\,,\,\ra_{\ggo}\qquad 
 \la\,,\,\ra_{\vv\times \vv}= \la\,,\,\ra_{\vv}\qquad \la \ggo, \vv\ra=0$$
 then the corresponding simply connected 2-step nilpotent Lie group $(N, \la\,,\,\ra)$, being $\la\,,\,\ra$ 
  the left invariant metric
 induced by $\la\,,\,\ra$ above,  is a naturally reductive pseudo Riemannian
 space.  
 
The converse holds whenever the center of $\nn$ is non degenerate and $j$ (defined
as in (\ref{br})) is faithful. }

\vskip 2pt

 The previous result empores the
understanding of some geometrical features such as the isometry group
-Proposition 3.5- and it sets up the construction of new examples, in particular by describing some naturally reductive metrics in the Heisenberg Lie group $H_{2n+1}$.

We also bring into consideration 2-step nilpotent Lie groups furnished  
with a bi-invariant metric in order to check geometrical and 
algebraic structure differences between  metrics for which the center is either  degenerate or either non degenerate. 
In fact bi-invariant metrics offer examples of  flat pseudo Riemannian metrics for which the e  
isometry group contains the group of orthogonal automorphisms as a proper 
subgroup.  Another application of bi-invariant metrics promotes the construction of
pseudo Riemannian naturally reductive compact spaces.

 \section{On 2-step nilpotent Lie groups with a left invariant pseudo Riemannian
 metric}
 
 In this section we show    suitable decompositions of the Lie algebra corresponding to a 2-step nilpotent
 Lie group equipped with a left invariant pseudo Riemannian metric. We are
 mainly interested here in those metrics for which the center is non degenerate, a fact that determines unambiguously the decomposition. 
 
 A {\em metric}  on a real vector space $\vv$ is a non 
 degenerate symmetric bilinear  form $\la\,,\,\ra:\vv \times \vv \to \RR$. 
 Whenever  $\vv$ is the Lie algebra of a given Lie group $G$, by identifying 
 $\vv$ with the set of left invariant vector fields on $G$, the metric 
 induces by mean of the left translations, a pseudo
  Riemannian metric tensor on
 the corresponding Lie group. Conversely any left invariant pseudo
 Riemannian metric on $G$ is completely determined by its value at the identity
 tangent space $T_eG$.
  
  Let  $(N, \la\,,\,\ra)$ denotes a 2-step nilpotent Lie group
endowed with a left invariant pseudo Riemannian metric. There exist several
ways to describe the structure of the corresponding
Lie algebra $\nn$. The main difficult yields on the existence of degenerate
subspaces as one can see below. 

\vskip 3pt

{\bf (a)}  If the center is degenerate,
the null subspace is defined uniquely as
$$\uu=\{x\in \zz \, \mbox{ such that }\, \la x, z\ra=0 \quad \forall z\in \zz\}$$
and therefore the center of $\nn$ decomposes as a direct sum of vector subspaces 
$$\zz =\uu \oplus \tilde{\zz}$$
where $\tilde{\zz}$ is a complementary subspace of $\uu$ in $\zz$ and it 
is easy
to prove that the restriction of the metric to $\tilde{\zz}$ is non 
degenerate. 
Moreover it is possible to find a isotropic subspace $\vv\subset \nn$ such that $\vv \cap
\zz=\{0\}$ and the metric on $\uu\oplus \vv$ is non degenerate. This subspace
$\vv$ is not well defined invariantly but once $\vv$ is fixed, one can take
$\tilde{\zz}$ as the portion of the center in $(\uu\oplus\vv)^{\perp}$ and to
complete the decomposition of $\nn$ as a orthogonal direct sum 
\begin{equation}\label{ortsum}
\nn=(\uu \oplus \vv) \oplus (\tilde{\zz}\oplus \tilde{\vv})
\end{equation}
in such a way that $(\uu\oplus \vv)^{\perp}=\tilde{\zz}\oplus \tilde{\vv}$ 
and (\ref{ortsum}) is a Witt decomposition. Note that $\tilde{\vv}$ is non degenerate.
Moreover  it is  possible to define a lineal map $j:\zz \to \End(\vv \oplus
\tilde{\vv})$  which
play a similar role to that one in the Riemannian case (see \cite{C-P1} for details).

In the last section of the present work, we show similar results for the
case of bi-invariant metrics. 
  
  \vskip 3pt
  
  {\bf (b)} Let $e_1, \hdots, e_p$ denote  a
basis of $\zz$. For any $u, v  \in \nn$, the Lie bracket can be 
written
 $$[u, v] =  \sum_{i=1}^p \la J_i u, v\ra e_i,
$$
where $J_i:\nn \to \nn$  are self adjoint endomorphisms with respect 
to $\la \,,\,\ra$ and $\zz=\cap_{i=1}^p ker J_i$. In fact, 
$[u,v]=\sum \omega_i(u,v) e_i$ where $\omega_i:\nn \times \nn \to \RR$ for
$i=1, \hdots p$, is a 
familly of skew symmetric bilinear 2-forms which represents the coordinates
of $[u,v]$ with respect to the fixed basis. 
Since the metric on $\nn$ is non degenerate, for every i there exists a endomorphism $J_i:\nn \to
\nn$ such that  $\omega_i(u,v)=<J_iu,v>$. 
  The endomorphisms $J_i$ are thus called the {\em structure endomorphisms}
    associated to $e_1 , . . . , e_p$ (see \cite{Bo}).

 \vskip 3pt

 Examples of pseudo
Riemannian 2-step nilpotent Lie groups  $N$ arise by considering the simply
connected Lie groups whose  Lie algebra can be constructed 
as follows.  
Let $(\zz, \la \,,\,\ra_{\zz})$ and $(\vv, \la \,,\,\ra_{\vv})$ denote vector
spaces endowed with (non necessarly definite) metrics. Let $\nn$ denote the
 direct sum as vector spaces
\begin{equation}\label{des2}
\nn=\zz \oplus \vv \qquad \quad\mbox{ direct sum }
\end{equation}
and let $\la\,,\,\ra$ denote the metric given by
\begin{equation}\label{met}
\la \,,\,\ra_{|_{\zz \times \zz}}=\la \,,\,\ra_{\zz}\qquad 
\la \,,\,\ra_{|_{\vv \times \vv}}=\la \,,\,\ra_{\vv} \qquad
\la \zz, \vv \ra=0.
\end{equation}

Let $j:\zz \to \End(\vv)$ be a linear map such that $j(z)$ is self adjoint with respect to
$\la \,,\,\ra_{\vv}$ for all $z\in \zz$.  Then $\nn$
 becomes a 2-step nilpotent Lie algebra if one defines a Lie bracket by
\begin{equation}\label{br}
\begin{array}{rcl}
[x,y] & = & 0 \quad \mbox{ for all }x\in \zz, y\in \nn\\
\la [u,v], x\ra & = & \la j(x) u,v\ra \qquad \mbox{ for } x\in \zz, u,v\in \vv.
\end{array}
\end{equation}

Conversely, let $\nn$ denote a 2-step nilpotent Lie algebra furnished with a metric
for which the center is non degenerate. Then $\nn$ can be decomposed into a
orthogonal direct sum  as in (\ref{des2}) being $\vv:=\zz^{\perp}$ and the 
Lie bracket on $\nn$ induces  self adjoint linear maps 
$j(x)$ for $x\in \zz$ given
by (\ref{br}).

\begin{prop} \label{p1} Let $(N,\la\,,\,\ra)$ denote a  simply connected 2-step nilpotent Lie group equipped with a left invariant pseudo Riemannian metric. If the center of $N$ is non
degenerate then its Lie algebra $\nn$ admits a orthogonal decomposition as in 
(\ref{des2}) and the corresponding Lie bracket can be obtained by (\ref{br}).
\end{prop}

This  includes the Riemannian case, that is,  when the metric $\la\,,\,\ra$ is
positive definite. 
In this situation, the inner product $\la\,,\,\ra_+$ produces  a decomposition of the center of the Lie algebra $\nn$   as a  orthogonal direct sum as vector spaces 
$$\zz=\ker j \oplus C(\nn)$$
and moreover $j$ is injective if and only if there is no Euclidean factor in the
De Rahm decomposition  of the simply connected Lie group $(N, \la\,,\,\ra_+)$ (see
\cite{Go}). This does not necessarly hold in the pseudo Riemannian case. 
 Below we show an example of a Lorentzian metric on a 2-step nilpotent
Lie algebra $\nn$, where the center is non degenerate and  such that 
$ker(j)=[\nn,\nn]$, so that a splitting as above is not possible.

\begin{exa} \label{exa1} Let $\RR \times \hh_3$ be the 2-step nilpotent Lie algebra spanned by
the vectors $e_1,e_2, e_3, e_4$ with the Lie bracket $[e_1, e_2]=e_3$. Define a
metric where the non trivial relations are
$$\la e_1, e_1\ra=\la e_2, e_2\ra=\la e_3, e_4\ra=1.$$
After (\ref{br}) one can verify that $j(e_3)\equiv 0$, while 
$$j(e_4)=\left(\begin{matrix}
0 & -1\\
1 & 0 \end{matrix} \right)
$$
Notice that $e_4\notin C(\RR \times \hh_3)$ and $ker j=\RR e_3=C(\RR\times
\hh_3)$, that is $ker j=C(\nn)$.
\end{exa}

 Let $\Or(\vv, \la\,,\,\ra_{\vv})$  denote the group of linear maps on $\vv$ which are isometries for
 $\la\,,\,\ra_{\vv}$ and whose Lie algebra $\sso(\vv,\la\,,\,\ra_{\vv})$ is the 
 set  of  linear maps on $\vv$ that are self adjoint  with respect to 
 $\la \,,\,\ra_{\vv}$. The next goal is to describe the group of isometries
 which plays an important role in the next section. Start with the next 
 result proved in \cite{C-P1}.

\begin{prop} \label{icp} Let $N$ denote a 2-step nilpotent Lie group endowed with a left invariant pseudo Riemannian metric, with respect to which the center is non degenerate. Then the group of isometries fixing the identity coincides with the group of orthogonal automorphisms of $N$.
\end{prop}

Denote by $H$ the group of orthogonal automorphisms and by  $N$ also  the 
subgroup of isometries consisting of left translations by elements of $N$.
 Consider the  isometries of the form $h n$ where $h\in H$ and $n\in N$, and denote it by $I_a(N)$. Then
$N$ is a normal subgroup of $I_a(N)$, $N\cap H=\{e\}$ and therefore  after (\ref{icp}) one has
$$I(N) = I_a(N)= H \ltimes N.$$

Whenever  $(N, \la\,,\,\ra)$ is  simply connected,   we do not distinguish between the group of automorphisms of
$N$ and of $\nn$. Thus one obtains that the group  $H$ is given by
\begin{equation}\label{oa}
H=\{(\phi, T)\in \Or(\zz, \la\,,\,\ra_{\zz}) \times \Or(\vv, \la\,,\,\ra_{\vv}):
Tj(x)T^{-1}=j(\phi x), \quad x\in \zz\}
\end{equation}
while its  Lie algebra $\hh=\Der(\nn)\cap \sso(\nn,\la\,,\,\ra)$ is 
\begin{equation}\label{sd}
\hh=\{(A,B)\in \sso(\zz,\la\,,\,\ra_{\zz}) \times \sso(\vv,\la\,,\,\ra_{\vv}):
[B,j(x)]=j(Ax),\quad x\in \zz\}.
\end{equation}
In fact, let $\psi$ denote an orthogonal automorphism of $(\nn, \la\,,\,\ra)$.
As automorphism $\psi(\zz)\subseteq \zz$ and since the decomposition 
$$ \nn = \zz \oplus \vv$$
is orthogonal then $\psi(\vv)\subseteq \vv$. Set $\phi:=\psi_{|_{\zz}}$ and
$T:=\psi_{|_{\vv}}$, thus 
$(\phi, T)\in \Or(\zz, \la\,,\,\ra_{\zz})\times \Or(\vv,\la\,,\,\ra_{\vv})$ such that
$$\begin{array}{rcl}
\la \phi^{-1}[u,v], x\ra &=&  \la [Tu,Tv],j(x) \ra \quad \mbox{if and only if}\\
\la j(\phi x) u, v \ra & = & \la j(x)Tu, Tv \ra
\end{array}
$$ 
which implies (\ref{oa}). By derivating (\ref{oa}) one gets (\ref{sd}).

\begin{prop}  Let $N$ denote a simply connected 2-step nilpotent Lie group endowed with a left 
invariant pseudo Riemannian metric, with respect to which the center is non 
degenerate. Then the group of isometries is
$$I(N) = H \ltimes N.$$
where $N$ denotes the set of left translations by elements of $N$ and $H$ 
the isotropy subgroup  is given by
(\ref{oa}) with  Lie algebra as in  (\ref{sd}).
\end{prop}

\begin{exa} \label{change} Let $\nn$ be a 2-step nilpotent Lie algebra equipped with an
inner product and denote it by $\la \,,\,\ra_+$.  
Let $J_z\in \sso(\vv, \la\,,\,\ra_+)$ denote the maps in (\ref{br}) for the inner 
product. 

We shall  consider a non definite metric $\la\,,\, \ra$ on $\nn$ by 
changing the sign of the metric on the center ${\zz}$; thus the metric
on $\vv$ remains invariant and we take 
$$\la z_i, z_j\ra=-\la z_i, z_j\ra_+\qquad \mbox{ for } z_i, z_j\in \zz \qquad
\mbox{ and } \qquad \la \zz, \vv\ra=0.$$
 By (\ref{br}) the maps $j(z)$ for the metric $\la\,,\,\ra$ on $\nn$ are
$$-\la z, [u,v]\ra_+= -\la J(z) u,v\ra_+ =\la j(z)u, v\ra=\la z,
[u,v]\ra,\quad \mbox{ for } z\in \zz$$
that is $j(z)=-J(z)$ for every $z\in \zz$.

\vspace{.3cm}

We work out an example on the Heisenberg Lie group $H_3$. This is the 
simply connected Lie group whose Lie algebra is $\hh_3$  which 
is spanned by the vectors 
$e_1, e_2, e_3$, with  the non trivial Lie bracket relation $[e_1, e_2]=e_3$. 
The canonical left invariant metric $(\la\,,\,\ra_+)$ is that one obtained by declaring the basis 
above to be orthogonal and the map 
$J(e_3)$ for $\la\,,\,\ra_+$ is
$$\left( \begin{matrix} 0 & -1 \\  1 & 0
\end{matrix} \right)
$$

A Lorentzian metric $\la\,,\,\ra$ is obtained on  $H_3$ by changing the sign of the 
canonical  metric on the center. 
Kaplan showed that $(H_3,\la\,,\,\ra_+)$  is naturally
reductive (\cite{Ka}).  In the next sections we shall see that 
$(H_3,\la\,,\,\ra)$ and generalizations of it, are  also naturally 
reductive. 

By (\ref{oa})  the group of isometries for any of these both metrics is 
$(\RR \times O(2))\ltimes H_3$,  where the action of the isotropy group is 
given by $(\lambda, A)\cdot (z+v)=\lambda z + Av$ for $z\in \zz$ and 
$v\in \vv = span\{e_1,e_2\}$, $\lambda \in \RR$ and $A\in O(2)$. 
 
\end{exa}

\begin{defn} A homogeneous manifold $M$
is said to be {\em naturally reductive} if there is a transitive Lie group of isometries
$G$ with Lie algebra $\ggo$  and there exists a subspace
$\mm\subseteq \ggo$ complementary to $\hh$,  the Lie algebra of the isotropy group 
$H$, in $\ggo$ such that
$$\Ad(H)\mm \subseteq \mm \qquad \mbox{ and }\qquad 
\la [x,y]_{\mm}, z\ra + \la y, [x,z]_{\mm} \ra=0 \qquad \mbox{ for all } x, y,
 z\in \mm.$$
\end{defn}
Frequently we will say that a metric on a homogeneous space $M$ is naturally reductive even though it is not naturally with respect to a particular transitive group of isometries (see Lemma 2.3 in \cite{Go}).

 For naturally reductive metrics the geodesics passing through $m\in M$ are of the form  $$\gamma(t)=\exp(t x) \cdot m\qquad \quad \mbox{ for some }x\in \mm.$$

A point $p$ of a pseudo Riemannian manifold is called a {\em pole} provided the exponential map $exp_p$ is a diffeomorphism. Furthemore if $o$ is a pole of the naturally reductive pseudo Riemannian manifold $G/H$, then the map $(x, h) \to \exp(x) h$ is a diffeomorphism of $\mm \times H\to G$ \cite{ON} Ch. 11.

  Indeed pseudo Riemannian
 symmetric spaces are naturally reductive. Examples of naturally
 reductive spaces arise from Lie groups equipped with a bi-invariant metric,
 which could exist for nilpotent ones. In the Riemmannian case, if
 a nilmanifold $N$ admits a naturally reductive metric, then $N$ is at most
 2-step nilpotent \cite{Go}.

\section{Naturally reductive metrics with non degenerate center: a
characterization}

In this section we achieve a  characterization of naturally reductive
pseudo Riemannian simply connected 2-step nilpotent Lie groups with non degenerate center 
 by studying the set of maps $j(z)$
$z\in \zz$ defined in (\ref{br}), showing that they build a subalgebra of the
Lie algebra of the isotropy  group $H$.

 \begin{lem} \label{l1} Let $(\nn, \la\,,\,\ra)$ denote a 2-step nilpotent Lie
  algebra  equipped with a metric for which its center $\zz$ is non degenerate
   and assume $j$ is injective. Let
 $\hh=\sso(\nn,\la\,,\,\ra) \cap \Der(\nn)$ denote the  Lie subalgebra of the group of 
 isometries
 fixing the identity element in the corresponding simply connected Lie group 
 $N$.  Then
 \vskip 3pt
 
 {\rm i)} $\hh$ leaves each of $\zz$ and $\vv$ invariant,
 
 \vskip 2pt
 
 {\rm ii)} For $\phi\in \hh$,
 $$\phi_{|_{\zz}} = j^{-1} \circ \ad_{\sso(\vv)}\phi_{|_{\vv}}\circ j.$$
 In particular $\phi \to \phi_{|_{\vv}}$ is an isomorphism of $\hh$ onto a
 subalgebra of $\sso(\vv,\la\,,\,\ra_{\vv})$.
 
 \vskip 2pt
 
 {\rm iii)} Let $\phi \in \sso(\vv,\la\,,\,\ra_{\vv})$. Then $\phi$ extends to an 
 element of $\hh$
 if and only if $[\phi, j(\zz)]\subseteq J(\zz)$ and  
 $j^{-1} \circ \ad_{\sso(\vv)}\phi_{|_{\vv}}\circ j  \in \sso(\zz,\la\,,\,\ra_{\zz})$.
 \end{lem}
\begin{proof} i) is easy to prove.  We shall
show (ii) and (iii). Let $A\in \sso(\zz,\la\,,\,\ra_{\zz})$ and $B\in
\sso(\vv,\la\,,\,\ra_{\vv})$, the linear map
$\phi$ which agrees with  $(A, B)\in \zz \oplus \vv$ lies in $\hh$ if and only
 if 
 $$\la j(Ax)u, v\ra=\la (B j(x)-j(x)B)u, v\ra \qquad \mbox{ for }x\in \zz,
 \,u,v\in \vv$$
  which is equivalent to $j(A(x))=[B, j(x)]$ the last one denotes the Lie
  bracket in $\sso(\vv,\la\,,\,\ra_{\vv})$ and since $j$ was assumed injective one gets
  $A=j^{-1}\circ\ad_{\sso(\vv)}(B) \circ j$.
\end{proof}

The proof of the next theorem coincides with that one given by C. Gordon in \cite{Go}. For the sake of
completeness we include it here. However the consequences are quite
different from the Riemannian situation.
 
\begin{thm} \label{t1} Let $(N, \la\,,\,\ra)$ denote a 2-step simply connected Lie group 
equipped with a left invariant pseudo Riemannian metric such that the center is non
degenerate and assume $j$ is injective. 
Then the metric is naturally reductive with respect to $G=H\ltimes N$
being $H$ the group of orthogonal automorphisms,  
if and only if
\vskip 3pt

{\rm (i)} $j(\zz)$ is a Lie subalgebra of $\sso(\vv,\la\,,\,\ra_{\vv})$ and 

\vskip 2pt

{\rm (ii)} $[j(x),j(y)] =  j(\tau_x y)$ where $\tau_x\in
\sso(\zz,\la\,,\,\ra_{\zz})$ for any 
$x\in \zz$.
\end{thm}

\begin{proof} Let $\ggo=\hh \ltimes \nn$ be the Lie algebra of $G=H\ltimes N$
 and assume $N$ is naturally reductive with respect to $\ggo=\hh \oplus \mm$.
 Set $\pi:\nn \to \hh$ so that
 $$\mm=\{ x + \pi(x): x\in \nn\}.$$
 The condition for natural reductivity says
 $$
\la [x+\pi(x),y+\pi(y)]_{\mm}, z+\pi(z)\ra_{\mm} =- \la y+\pi(y),
[x+\pi(x),z+\pi(z)]_{\mm} \ra_{\mm}$$
where $\la\,,\,\ra$ is the pseudo Riemannian metric on $\mm$, so that the
previous equality can be interpreted on $\nn$ as
\begin{equation}\label{onn}
\la [x,y]+\pi(x)y-\pi(y)x, z\ra=-\la y, [x,z]+\pi(x) z-\pi(z)x\ra.
\end{equation}
where  $\pi(x)$ is view as a linear operator on $\nn$ and one writes
$\pi(x)y=[x,y]$ when $x, y\in \nn$. Since $\pi(x)\in \sso(\nn, \la\,,\,\ra)$ the terms
involving $\pi(x)$ cancel and (\ref{onn}) yields
\begin{equation}\label{e11}
\ad(y)^*z+\ad(z)^*y = \pi(y) z+\pi(z) y\quad \mbox{ for all }y, z\in \nn.
\end{equation}
Since $[\hh, \nn]\subseteq \nn$ and $[\hh,
\mm]\subseteq \mm$, one has
$$[\pi(x), y+\pi(y)]=\pi(x) y +[\pi(x),\pi(y)]\in \mm$$
and therefore
\begin{equation}\label{e12}
\pi(\pi(x)y)=[\pi(x), \pi(y)] \quad \mbox{ for all } x,y \in \nn.
\end{equation}
If $z\in \zz$ and $y\in \vv$, $\ad(z)^*y=0$ and (\ref{e11}) says
\begin{equation}\label{e13}
j(z)y=\pi(y)z+\pi(z)y.
\end{equation}

But $\pi(y)z\in \zz$ and $\pi(z)y\in \vv$, so (\ref{e13}) implies
$$\pi(z)_{|_{\vv}}= j(z)\in \sso(\vv,\la\,,\,\ra_{\vv})\quad \mbox{ for every } z\in \zz.$$
It then follows that 
$$[j(x), j(\zz)]\subset j(\zz) \quad \mbox{ and } \quad [j(x),
j(y)]=j(\tau_x y)\quad \mbox{ for } \tau_x\in
\sso(\zz,\la\,,\,\ra_{\zz}),\quad x,y\in \zz.$$

Conversely if (i) and (ii) hold, extend $j(x)$ to  an  element  $\pi(x)$ of
$\hh$ such that  the restriction of $\pi(x)$ to $\zz$ is given by the left
hand side of (ii). Extend $\rho$ as a linear map of $\nn$ by declaring
$\pi_{|{\vv}}\equiv 0$.  We claim (\ref{e12}) hold for all $x, y\in \nn$.  In fact it
is easy to verify it if at least one of $x,y\in \vv$. Assume $x,y \in \zz$, then
$$\pi(\pi(x)y)_{|_{\vv}}=j(j^ {-1}[j(x),j(y)])=[j(x),j(y)]$$
and therefore (\ref{e12}) is true after (\ref{l1}) ii). Define
$$\mathfrak l=\pi(\nn), \quad \mm=\{ x + \pi(x) : x \in \nn\}, \quad \mbox{
and } \quad \kk=\mathfrak l \oplus \mm.$$
By (\ref{e12}) $\mathfrak l$ is a Lie subalgebra of $\hh$ and $[\mathfrak l,
\mm]\subseteq \mm$ and since $\kk=\mathfrak l \oplus \nn$, $\kk$ is a Lie subalgebra of $\ggo$.

We assert that (\ref{e11}) is valid. This can be easily checked whenever at least
one of $x, y\in \vv$. If both $x, y\in \zz$ the left-hand side of (\ref{e11}) is
zero. The right-hand side lies in $\zz \cap ker(\pi)$, but
$\ker(\pi)=\ker(j)$ and since $j$ is injective one has $\zz \cap
ker(\pi)=\{0\}$, which proves (\ref{e11}). By following the argument preceding
(\ref{e11}) backwards, one can see that $M$ is naturally reductive with respect
to $\kk$.
\end{proof}

If $\hh$ is a Lie subalgebra of $\End(\vv)$ such that $\hh\subseteq \sso(\vv,
\la\,,\,\ra_{\vv})$ then we call $\la\,,\,\ra$ an {\em $\hh$-invariant metric}. 

In the conditions of  Theorem (\ref{t1}) it follows that if $(N, \la\,,\,\ra)$ is naturally
reductive  then the bilinear map $\tau$ defines a Lie algebra
structure on $\zz$ and the map $j:\zz \to \sso(\vv,\la\,,\,\ra_{\vv})$ becomes a real
representation of the Lie algebra $(\zz, \tau)$. Furthermore the metric
on $\vv$ is $j(\zz)$-invariant and since $\tau_x\in \sso(\zz, \la\,,\,\ra_{\zz})$ the
metric on $\zz$ is ad($\zz$)-invariant, where $\ad$ denotes the adjoint
representation of $(\zz, \tau)$.

Conversely let $\ggo$ be a real Lie algebra endowed with an ad($\ggo$)-invariant
metric $\la \,,\, \ra_{\ggo}$ and let $(\pi, \vv)$ be a faithful representation
of $\ggo$ endowed with a $\pi(\ggo)$-invariant metric $\la\,,\,\ra_{\vv}$ and
without trivial subrepresentations, that is, $\bigcap_{x\in \ggo}ker \pi(x)=\{0\}$.
Define a 2-step nilpotent Lie algebra structure on the vector space underlying
$\nn=\ggo \oplus \vv$  by the following bracket
\begin{equation}\label{brac}
\begin{array}{ll}
[\ggo,\ggo]_{\nn}=[\ggo, \vv]_{\nn} =0 & [\vv, \vv] \subseteq \ggo \\ \\
\la [u,v], x\ra_{\ggo} = \la \pi(x) u, v\ra_{\vv}& \forall x\in \ggo, u, v\in
\vv.
\end{array}
\end{equation}
and equip $\nn$ with the metric obtained as the product metric
\begin{equation}\label{metric} 
\la \,,\,\ra_{|_{\ggo \times \ggo}}=\la \,,\, \ra_{\ggo}\qquad \la
\,,\,\ra_{|_{\vv \times \vv}}=\la \,,\, \ra_{\vv} \qquad \la \ggo, \vv\ra=0.
\end{equation}

Take $N$ the simply connected 2-step nilpotent Lie group with Lie algebra 
$\nn$ and endow it with the left invariant metric determined by $\la \,,\,\ra$.

Since $(\pi, \vv)$ has no trivial subrepresentations, the center of $\nn$ 
coincides with $\ggo$. Moreover $\vv$ is its orthogonal complement  and the 
transformation $j(x)$ defined as in (\ref{br}) is precisely $\pi(x)$ for all
$x\in \ggo$. Since $(\pi, \vv)$ is faithful, the commutator of $\nn$ is $\ggo$:
 $C(\nn)=\ggo$. Since the set $\{\pi(x)\}_{x\in \ggo}$  is a Lie subalgebra of 
 $\sso(\vv,\la\,,\,\ra_{\vv})$ we conclude that $(N, \la\,,\,\ra)$ is 
 naturally reductive.

\begin{thm}\label{t2} Let $\ggo$ denote a Lie algebra equipped with an 
ad-invariant metric $\la\,,\,\ra_{\ggo}$ and let $(\pi, \vv)$ be a real 
faithful representation of $\ggo$ without trivial subrepresentations and 
endowed with a $\pi(\ggo)$-invariant metric $\la\,,\,\ra_{\vv}$. 
Let $\nn$ be the Lie algebra $\nn=\ggo \oplus \vv$ direct sum of vector spaces,
together with the Lie bracket given by  (\ref{br}) and furnished with the metric 
 $\la\,,\,\ra$ as in (\ref{metric}). Then the corresponding simply connected 
 2-step nilpotent Lie group 
 $(N,\la \,,\,\ra)$, being $\la\,,\,\ra$ the induced left invariant metric, is a naturally reductive pseudo Riemannian
 space.  
 
 The converse holds whenever the center of $N$ is non degenerate and $j$ is
 faithful. \end{thm}

\begin{rem} Suppose the representation $(\pi, \vv)$ of $\ggo$ is not 
faithful. Thus 
$$z\in ker \pi \Longleftrightarrow \la z, [u,v]\ra=0 \quad \forall u,v \in
\vv $$
$\Longrightarrow z \in C(\nn)^{\perp}$. Since the metric on the center $\ggo$ is non definite,  $ker \pi \cap C(\nn)$ coul be non trivial, so that the sum as vector spaces $ker \pi+C(\nn)$ is not necessarly direct. 

When $\pi$ has some trivial subrepresentation,  
$$u\in \cap_{x\in \ggo} \pi(x) \Longleftrightarrow \la \pi(x)u, v\ra=0 \quad
\forall v\in \vv,$$
$\Longrightarrow\la x, [u,v]\ra=0$ for all $x\in \ggo$, thus $[u,v]=0$ for
all $v\in \vv$ which says   $u\in
\zz(\nn)$. Hence $\ggo \subsetneq \zz(\nn)$. 
\end{rem}

\begin{rem} While in the Riemannian case, the condition of the metric to be 
positive definite says that $\ggo$ must be compact, in the pseudo Riemannian 
case the statement above imposes the restriction on $\ggo$ to carry an
ad-invariant metric. See the next example. 
\end{rem} 

\begin{exa} \label{exad} The Killing form on any semisimple Lie algebra is an ad-invariant
metric. 

Any Lie algebra $\ggo$ can be embedded into a Lie algebra which admits an
ad-invariant metric. In fact, the  cotangent $\ct^*\ggo=\ggo
\ltimes_{coad}\ggo^*$, being $coad$ the coadjoint representation,  admits a neutral ad-invariant metric which is given by:
$$\la (x_1, \varphi_1),(x_2,\varphi_2)\ra=\varphi_1(x_2)+\varphi_2(x_1)\qquad
\qquad x_1, x_2\in \ggo,\quad \varphi_1, \varphi_2\in \ggo^*.$$
Notice that both $\ggo$ and $\ggo^*$ are isotropic subspaces. 
\end{exa}

\vskip 3pt

A {\em data set} $(\ggo, \vv, \la\,,\,\ra)$ consists of 

(i) a Lie algebra $\ggo$  equipped with a $\ad$-invariant  metric
$\la\,,\,\ra_{\ggo}$,

(ii) a real faithful representation of $\ggo$: $(\pi, \vv)$, without trivial
subrepresentations, 

(iii) $\la\,,\,\ra$ is a $\ggo$-invariant metric on $\nn=\ggo\oplus
\vv$, i.e. $\la\,,\,\ra_{|_{\ggo\times \ggo}}=\la\,,\,\ra_{\ggo}$ is 
ad($\ggo$)-invariant and 
$\la\,,\,\ra_{|_{\vv\times \vv}}$ is $\pi(\ggo)$-invariant and $\la \ggo,\vv\ra=0$.

A data set $(\ggo, \vv, \la\,,\,\ra)$ determines a 2-step nilpotent Lie group
denoted by $N(\ggo, \vv)$ whose Lie algebra is the underlying vector space $\nn=\ggo
\oplus \vv$ with the Lie bracket defined by (\ref{brac}). Extend the
metric on $\ggo$ by left translations after identifying 
$\nn\simeq T_eN(\ggo, \vv)$, so that $N(\ggo, \vv)$ becomes a naturally reductive
pseudo Riemannian 2-step nilpotent Lie group   (\ref{t2}).

We study the isometry group in this case. Let $\hh$ denote the Lie algebra of
the isometries fixing the identity element; by (\ref{sd}) an element $D\in \hh$
is a self adjoint derivation which can be written as
$D=(A,B)\in \sso(\ggo,\la\,,\,\ra_{\ggo}) \times \sso(\vv,\la\,,\,\ra_{\vv})$ such that 
$$ B\pi(x) -\pi(x) B =\pi(Ax),\quad \forall x\in \ggo.$$
Denote by $[\,,\,]_{\nn}$ the Lie bracket on $\nn$ and by $[\,,\,]$ the Lie
brackets on $\ggo$ and $\End(\vv)$. Then
$$\begin{array}{rcl}
\pi(A[x,y]) & = & B\pi([x,y])-\pi([x,y])B = B[\pi(x),\pi(y)]-[\pi(x),\pi(y)]B \\
& = & [B, [\pi(x),\pi(y)]]=[[B,\pi(x)]+[\pi(x),[B,\pi(y)]]\\
& = & [\pi(Ax),\pi(y)]+[\pi(x),\pi(Ay)]=\pi([Ax,y]+[x,Ay]).
\end{array}
$$
Since $\pi$ is faithful then
$$A[x,y]=[Ax,y]+[x,Ay]\qquad \quad\mbox{ for all } x,y \in \ggo, $$
that is, $A\in \Der(\ggo) \cap \sso(\ggo,\la\,,\,\ra_{\ggo})$. 

\begin{prop} \label{pi} The group of isometries fixing the identity on a naturally reductive pseudo Riemannian 2-step nilpotent Lie group $N(\ggo, \vv)$ as in (\ref{t2}) has Lie algebra
$$\hh=\{(A,B)\in (\Der(\ggo)\cap \sso(\ggo, \la\,,\,\ra_{\ggo}))\times \sso(\vv,\la\,,\,\ra_{\vv})\,:\, [\pi(x),B]=\pi(Ax)\quad \forall x\in \ggo\}.$$
\end{prop}

Whenever $\ggo$ is semisimple, the ad-invariant metric on $\ggo$ is the
Killing form; therefore any self adjoint derivation of $\ggo$ is of the form $\ad(x)$ 
for some $x\in \ggo$ . In this case one can consider $\ggo
\subset \hh$ where the action is given as
$$ x \cdot (z + v) = \ad(x) z + \pi(x) v\qquad x\in \ggo, \,z+ v\in \nn$$
being $\ad(x)$  the adjoint map on the semisimple Lie algebra $\ggo$.
Thus an element $D=(A,B)\in \hh$ is of the form
$$(A,B)=(\ad(x), \pi(x))+(0,B')\qquad x\in \ggo$$
with $B'=B-\pi(x)\in \End_{\ggo}(\vv)\cap \sso(\vv,
\la\,,\,\ra_{\vv})=\ee_{\ggo}$, where 
$\End_{\ggo}(\vv)$ denotes the set of intertwinning operators of the
representation $(\pi,\vv)$ of $\ggo$. Since $\ggo$ and $\ee_{\ggo}$ commute, 
then $\hh=\ggo\oplus \ee_{\ggo}$ is a direct sum of Lie algebras, here we
identify $\ggo$ with the set $\{(\ad(x), \pi(x)): x\in \ggo\}\subseteq \hh$.
This argues the following result.

\begin{cor} \label{coro} In the conditions of (\ref{pi}) with data set $(\ggo,\vv,
\la\,,\,\ra)$ for $\ggo$ semisimple the group of
isometries fixing the identity element is
$$H=G \times U\qquad \qquad U=\End_{\ggo}(\vv) \cap \Or(\vv,
\la\,,\,\ra_{\vv}).$$
\end{cor} 
\begin{proof} By (\ref{oa}) we have that
$$H=\{(\phi, T)\in \Or(\ggo, \la\,,\,\ra_{\ggo}) \times \Or(\vv, \la\,,\,\ra_{\vv}):
T\pi(x)T^{-1}=\pi(\phi x), \quad x\in \ggo\}.$$
Hence $\phi=\pi^{-1}\circ \Ad(T)\circ \pi\in \Aut(\ggo)$. Since $\ggo$ is
semisimple any automorphism of $\ggo$ is an inner automorphism, thus there
exist $g\in G$ such that $\phi=\Ad(g)$. By the paragraph above, 
 $(Ad(g),\pi(g))\in H$ and therefore $\pi(g)^{-1} T\in U$. Hence 
$$(\phi,T)=(\Ad(g), \pi(g)) \cdot (I, \pi(g)^{-1} T),$$
which says $H=G\times U$.
\end{proof}

\begin{rem} Compare with  \cite{La}.
\end{rem}

\section{Geometry and Examples of naturally reductive 2-step nilmanifolds with non degenerate center}

The aim of this section is twofold. In the first part we write explicitly
some geometric features to bring into the proof of (\ref{icp}), while in the second part we show examples of
naturally reductive pseudo Riemannian 2-step nilpotent Lie groups with non
degenerate center.

Recall that a 2-step nilpotent Lie algebra $\nn$ is said to be {\em non singular} if
 $\ad(x)$  maps $\nn$ onto $\zz$ for every $x\in \nn-\zz$. Suppose $\nn$ is equipped with a metric as
 in (\ref{des2}) then $\nn$ is non singular if and only if $j(x)$ is non
 singular for every $x\in \zz$. We shall say that a Lie group is non singular if its corresponding Lie algebra is non singular.

Whenever  $N$ is simply connected 2-step nilpotent the exponential map 
$\exp:\nn \to N$ produces global coordinates. In terms of this map the product on
$N$ can be obtained by
$$\exp(z_1+v_1) \exp(z_2+v_2)= \exp(z_1+z_2+\frac12 [v_1,v_2] + v_1+v_2) \quad\mbox{ for  } z_1, z_2\in \zz, \, v_1, v_2\in \vv.$$

We shall study the geometry of 2-step nilpotent  Lie groups  when they are 
endowed with a left invariant (pseudo Riemannian) metric $\la\,,\,\ra$ with respect to which the 
center is non degenerate. In the
Riemannian case a deep study of the geometry can be found in the works of P. Eberlein \cite{Eb1, Eb2}.

 The covariant derivative $\nabla$ is left invariant, hence one can see $\nabla$ as a bilinear form on $\nn$ getting the formula
\begin{equation}\label{nabla}
 \nabla_x y= \frac12([x,y]-\ad(x)^*y -\ad(y)^*x) \qquad \quad \mbox{ for }
 x,y \in \nn,
\end{equation}
where $\ad(x)^*$ denotes the adjoint of $\ad(x)$. By writing  this
explicitly one obtains
\begin{equation}\label{nablaex}
\begin{array}{rcll}
\nabla_x y & = & \frac 12[x,y] & \mbox{ for all } x,y \in \vv\\
\nabla_x y= \nabla _y x & = & -\frac12 j(y) x & \mbox{ for all } x\in \vv, y\in
\zz \\
\nabla_x y & = & 0 & \mbox{ for all } x, y\in \zz
\end{array}
\end{equation}
 
 Since translations on the left are isometries, to describe the geodesics of $(N, \la\,,\,\ra)$ it suffices to describe those 
 geodesics  that begin at $e$ the identity of $N$. Let $\gamma(t)$ be a curve 
 with $\gamma(0) = e$, and let $\gamma'(0) = z_0+v_0 \in \nn$, where $z_0\in
 \zz$ and $v_0\in \vv$. In exponential coordinates we write
       $$\gamma(t)=\exp(z(t)+v(t)), \quad \mbox{ where }     z(t)\in \zz, \,
       v(t)\in \vv\quad  \mbox{ for all $t$  and } z'(0)=z_0,\, v'(0)=v_0.$$
       
   The curve $\gamma(t)$ is a geodesic if and only if the following equations
are satisfied:
 \begin{eqnarray} \label{egeo}
  v''(t) & = & j(z_0) v'(t) \mbox{ for all }t \in \RR \\
  z_0 & \equiv & z'(t) + \frac12 [v'(t), v(t)]  \mbox{ for all }t \in \RR
 \end{eqnarray}  
             
  These equations were derived by A. Kaplan in \cite{Ka} to study 
  2-step nilpotent groups N of Heisenberg type, but the proof is valid in
  general for  2-step nilpotent Lie groups equipped with a left invariant 
  pseudo Riemannian metric where the center is  non degenerate as noted in
  \cite{Ge} and \cite{Bo}. 
 
 Let $\gamma(t)$ be a geodesic of $N$ with $\gamma(0) = e$. Write 
 $\gamma'(0) = z_0+v_0$, where $z_0\in \zz$ and $v_0\in \vv$ and identify
 $\nn=T_eN$.  Then
 \begin{equation}\label{geo}
 \gamma'(t) = dL_{\gamma(t)}(e^{tj(z_0)} v_0 + z_0) \qquad \mbox{ for all }
 t \in \RR
 \end{equation}             
where $e^{tj(z_0)}= \sum_{n=0}^{\infty} \frac{t^n}{n!} j(z_0)^n$. In fact, 
write $\gamma(t)=exp (z(t)+v(t))$, where $z(t)$ and $v(t)$ lie in
$\zz$ and $\vv$ respectively for all $t \in \RR$. By using the previous
equations (\ref{egeo}) one has
$$\begin{array}{rcl}
\gamma'(t) &=&  d\exp_{z(t)+v(t)}(z'(t)+v'(t))_{z(t)+v(t)} \\
& = &  dL_{\gamma(t)}(z'(t) + \frac12 [v'(t), v(t)]+ v')\\
& = &  dL_{\gamma(t)}(z_0+ v').
\end{array}
$$
Now by integrating the first equation of (\ref{egeo}) one gets
$v'(t)=e^{tj(z_0)} v_0$ which proves (\ref{geo}).

For $x,y$ elements in $\nn$ the curvature tensor is defined by
$$R(x,y)=[\nabla_x, \nabla_y]-\nabla_{[x,y]}.$$
Using (\ref{nablaex}) one gets
\begin{equation}\label{cu}
R(x,y)z = \left\{
\begin{array}{ll}
 \frac12 j([x,y])z -\frac14j([y,z])x+\frac14 j([x,z])y& \mbox{
for } x,y,z \in \vv, \\ \\
 -\frac14 [x,j(y)z] &  \mbox{ for } x, y \in \vv, \, z\in \zz,\\ \\ 
 -\frac14 [x,j(z) y]+\frac14 [y,j(z) x] & \mbox{ for } x, z \in \vv, \, y\in \zz,\\ \\
 -\frac14 j(y)j(z) x & \mbox{ for } x\in \vv, y, z \in \zz,\\ \\
 \frac14 [j(x),j(y)]z & \mbox{ for } x,y \in \zz, z\in \vv,\\\\
 0 & \mbox{ for } x,y, z\in \zz.
\end{array} \right.
\end{equation}

Let $\Pi\subseteq \nn$ denote a non degenerate plane and let $Q$ be given by
$$Q(x,y)=\la x,x \ra \la y,y\ra-\la x,y \ra^2.$$
 The non degeneracy property is equivalent to ask $Q(v,w)\neq 0$ for one -hence every- basis $v,w\in \Pi$ \cite{ON}.
The sectional curvature of $\Pi$ is the number $K(x,y):=\la R(x,y)y, x\ra
/Q(x,y)$, which is  independent of the choice of the basis. 
Now take a orthonormal
basis for $\Pi$, that is a linearly independent set $\{x,y\}$ such that
$\la x,y\ra=0$ and $\la x, x \ra =\pm 1$ and $\la y, y \ra=\pm 1$. 

After (\ref{cu}) one obtains
\begin{equation}\label{sect}
K(x,y)  = \left\{
\begin{array}{ll}
-\frac{3 \varepsilon_1 \varepsilon_2}4 \la [x,y], [x,y]\ra & \mbox{
for }x,y\in \vv \\ \\
 \frac{\varepsilon_1 \varepsilon_2}4 \la j(y)x, j(y)x \ra & \mbox{
for }x\in \vv, y \in \zz,\\ \\
0 & \mbox{ for }x,y \in \zz
\end{array} \right.
\end{equation}
being $\varepsilon_1:=\la x,x \ra$ and $\varepsilon_2:= \la y, y\ra$.

\vskip 3pt

The Ricci tensor is given by 
$Ric(x,y)={\rm trace}(z \to R(z,x)y), z\in \nn$ for arbitrary elements 
$x,y\in \nn$. 

\begin{prop}\label{p2} Let $\{z_i\}$ denote a orthonormal basis of $\zz$ and $\{v_j\}$ a orthonormal basis of $\vv$. It holds
$$Ric(x,y)=\left\{
\begin{array}{ll}
0 & \mbox{ for }x\in \vv, y\in \zz\\
\\
\frac12 \sum_i \varepsilon_i \la j(z_i)^2 x, y\ra & \mbox{ for } x, y\in \vv,\, \varepsilon_i=\la z_i, z_i\ra\\ \\
-\frac14 \sum_j {\varepsilon}_j \la j(x) j(y) v_j, v_j\ra & \mbox{ for } x,y \in \zz,\, {\varepsilon}_j=\la v_j, v_j\ra.
\end{array} \right.
$$
\end{prop}

Due to symmetries of the curvature tensor, the Ricci tensor 
is a symmetric bilinear form on $\nn$ and hence there exists a 
symmetric linear transformation $T:\nn \to \nn$ such that 
$Ric(x,y)=\la Tx,y\ra$ for all $x,y \in \nn$. $T$ is called the Ricci
transformation. Let $\{e_k\}$ denote a orthonormal basis of $\nn$; it holds
$$Ric(x,y)=\sum_k \varepsilon_k \la R(e_k,x)y, e_k \ra =\la -\sum_k
\varepsilon_k R(e_k,x) e_k, y\ra$$
which implies 
\begin{equation} \label{rt}
T(x)= -\sum_k \varepsilon_k R(e_k,x) e_k, \qquad \mbox{ being }
\varepsilon_k=\la e_k, e_k\ra.
\end{equation}
According to the results in (\ref{p2}) we have that $\zz$ and $\vv$ are
$T$-invariant subspaces and 
$$
T(x) = \left\{ \begin{array}{ll}
 \frac12 j(\sum_i
\varepsilon_i z_i)^2 x &  x\in \vv, \quad\varepsilon_i=\la z_i, z_i\ra \\ \\
 \frac14 \sum_j \varepsilon_j
[v_j, j(x)v_j] & x\in \zz \qquad \varepsilon_j=\la v_j, v_j \ra.
\end{array} \right.
$$
where $\{z_i\}$ and $\{v_j\}$ are orthonormal basis of $\zz$ and $\vv$
respectively.

 \begin{rem} The formulas above were used in \cite{C-P1} to prove (\ref{icp}).
 \end{rem}
 
 \begin{rem} For naturally reductive metrics  in the formulas
 above, replace the maps $j$ be the corresponding representation $\pi:\ggo
 \to \sso(\vv, \la\,,\,\ra_{\vv})$.
 \end{rem}
 
 Below we expose examples of naturally reductive metrics on 2-step
 nilpotent Lie groups. This is achieved by translating the data at the Lie
 algebra level to the corresponding simply connected Lie group by
 following  the key results provided in 
 (\ref{t2}). We  shall make use of 
 euclidean and semisimple Lie algebras in order to obtained ad-invariant
 metrics. For further details on Lie algebras with ad-invariant metrics see
 for instance \cite{F-S,M-R}. Concerning isometries between pseudo Riemannian
 2-step nilpotent Lie groups notice that orthogonal isomorphisms gives rise to isometries between the corresponding Lie groups (*).
 
\vskip 3pt 
 
\ri \,{\it Riemannian examples.} Naturally reductive Riemannian nilmanifolds arise by considering a data set with $\ggo$
compact. Recall that if $\ggo$ is compact then $\ggo=\kk \oplus \cc$ where
$\kk=[\ggo, \ggo]$  is a compact semisimple Lie algebra and $\cc$ is the 
center (see \cite{Wa}). In \cite{La} they were extended studied.

In the Riemannian case the converse of (*) above holds \cite{Wi}.

\vskip 2pt

\rii \,{\it Modified Riemannian.} Take any of those data sets corresponding to the positive definite case and  follow the ideas in (\ref{change}). Clearly
all requierements in (\ref{t2}) apply and so one can produce naturally
reductive pseudo Riemannian metrics of signature ($\dim \ggo, \dim \vv)$. 

Let $N(\ggo,\vv)$ denote a Riemannian naturally reductive nilmanifold
obtained from a data set $(\ggo, \vv, \la\,,\,\ra)$. Let
$\tilde{N}(\ggo,\vv)$ denote the pseudo Riemannian 2-step nilpotent Lie
group obtained by changing the sign of the metric on $\ggo$. Therefore by \cite{Wi}
$$N(\ggo,\vv)\simeq N'(\ggo',\vv') \qquad \Longleftrightarrow \qquad \nn(\ggo,\vv)\simeq
\nn(\ggo, \vv')$$
and this occurs if and only if there an isometric isomorphism $\phi:(\ggo,
\la\,,\,\ra_+) \to (\ggo', \la\,,\,\ra_+')$ and a isometry $T:(\vv,
\la\,,\,\ra_+) \to (\vv',\la,,\ra_+)$ such that
$$T\pi(x)T^{-1}=\pi'(\phi x) \qquad \mbox{ for all }x\in \ggo.$$
Clearly $\phi:(\ggo, -\la\,,\,\ra_+) \to (\ggo', -\la\,,\,\ra_+')$ is also a isometric isomorphism, so that the  corresponding simply connected Lie groups are  isometric. Thus one has  what follows. 

\begin{prop} If $N(\ggo,\vv)\simeq N'(\ggo',\vv')$ then $\tilde{N}(\ggo,\vv)\simeq \tilde{N}'(\ggo',\vv')$.
\end{prop}

In \cite{La}  detailed conditions to get the isometries $N(\ggo,\vv)\simeq N'(\ggo',\vv')$ were obtained.

\vskip 2pt

\riii \, {\it Abelian center.} Let $\RR^{2n}$ be equipped with a metric $B$, that is, $B$ is determined by a non singular symmetric linear map such that 
$$B(x,y)=\la b x, y\ra \qquad \la \,,\,\ra \mbox{ the canonical inner product on } \RR^{2n}.$$
 Let $t\in \sso(\RR^{2n}, B)$, that is $t$ may satisfy $t^*=-btb$ where $t^*$ denotes adjoint with respect to the canonical inner product on $\RR^{2n}$.
 
 Any non singular $t\in \sso(\RR^{2n}, B)$ gives rise to a faithful 
 representation of $\RR$ to $(\RR^{2n},B)$ without trivial 
 subrepresentations. Let $\nn$ be the vector space direct sum 
 $\RR z \oplus \RR^{2n}$ equipped with a metric $\la\,,\,\ra$ such that 
 $$\la z, \RR^{2n}\ra=0\qquad \la z, z \ra= \lambda\in \RR-\{0\} \qquad 
 \la\,,\,\ra_{\RR^{2n}}=B.$$
  Define a Lie bracket on $\nn$  by
 $$[z, y]=0\qquad \forall y\in \nn\qquad \mbox{ and }\quad \la [u,v], z\ra= B(t u,v) \quad u,v\in \RR^{2n}.$$
 According to (\ref{t2}) this Lie bracket makes of $\nn$ a 2-step nilpotent Lie algebra and the given metric is naturally reductive whenever the center is non degenerate. This Lie algebra is isomorphic to the Heisenberg Lie algebra.  
 
 Furthermore, the group of isometries  fixing the identity element has Lie algebra
 \begin{equation} \label{ih} \hh=\mathcal Z_{\sso(\RR^{2n}, B)} (t)
 \end{equation}
 where $\mathcal Z_{\sso(\RR^{2n}, B)} (t)$ denotes the centralizer of $t$
  in $\sso(\RR^{2n},B)$, which can be verified by applying Proposition (\ref{pi}). 
 
 In this way one gets naturally reductive metrics on the Heisenberg
  Lie group of dimension 2n+1. The converse also holds.
 
\begin{prop} Any left invariant   pseudo Riemannian metric on the Heisenberg
Lie group $H_{2n+1}$ for which
the center is non degenerate is naturally reductive. 

The isotropy group has Lie algebra $\hh$ as in (\ref{ih}). 
\end{prop}
\begin{proof}
Let $\hh_{2n+1}$
denote the Lie algebra of $H_{2n+1}$ and decompose it as a orthogonal direct sum
$\hh_{2n+1}=\RR z \oplus \vv$. Then the restriction of the metric to $\vv$ defines a metric $B$ of signature $(k,m)$. 
The map $j$ defined in
(\ref{br}) is indeed self adjoint  with respect to
$B:=\la\,,\,\ra_{|_{\vv\times\vv}}$ and it generates a subalgebra of 
 $\sso(\vv, B)$. Thus $z \to j(z)$ defines a faithful representation 
 without trivial subrepresentations since by $t:=j(z)$ one has
 $$tu=0 \Longleftrightarrow B( tu, v)=0 \quad \forall v\in \vv \Longleftrightarrow B(z,[u,v])=0 \quad \forall v\in \vv.$$
 But since the center is non degenerate then $[u,v]=0$ for all $v\in \vv$ which implies $u=0$.
 Indeed any non degenerate  metric on $\RR z$  is ad-invariant. 
Hence the statements of (\ref{t2}) are satisfied and the metric on $\hh_{2n+1}$ is naturally reductive. 
\end{proof}

\begin{exa} Let $\hh_3$ denote the Heisenberg Lie algebra of dimension three
with basis $e_1, e_2, e_3$ satisfying the Lie brackets $[e_1, e_2]=e_3$.
 Lorentzian metrics on $\hh_3$ with non degenerate center can be defined by
 $$\begin{array}{lrclcl}
 (1) & -\la e_3, e_3\ra & = & 1 & = & \la e_1, e_1\ra=\la e_2, e_2\ra \\
 (2) & \la e_3, e_3\ra & = & 1 & = & -\la e_1, e_1\ra=\la e_2, e_2\ra 
 \end{array}
 $$
 Thus in the basis $e_1,
 e_2$ the map $j_1(e_3)$ for the metric in (1) is represented  by the matrix
 $$\left( \begin{matrix} 0 & 1 \\ -1 & 0 \end{matrix} \right)$$
 (compare with (\ref{change})) while $j_2(e_3)$ for the metric (2) one has
 $$\left( \begin{matrix} 0 & 1 \\ 1 & 0 \end{matrix} \right).$$
 
 See \cite{Ge} for more results concerning Lorentzian metrics.
 \end{exa}

The construction on the Heisenberg Lie algebra, can be extended in the following
way.  Set $B$ a non degenerate symmetric bilinear form on $\RR^k$ and let $t_1,
\hdots t_l$ be commuting linear maps in $\sso(\RR^k, B)$ and such that
$\bigcap_i ker(t_i)=\{0\}$.

Set $\nn=\RR^l \oplus \RR^k$ direct sum of vector spaces, equipped $\RR^l$ with
any metric and $\nn$ with the product metric such that $\la \RR^l, \RR^k\ra=0$. 

The triple $(\RR^l, \RR^k, \la\,,\,\ra)$ is a data set which induces a naturally
reductive metric on the corresponding simply connected 2-step nilpotent Lie 
group with Lie algebra $\nn$.

\vskip 2pt \, {\it Semisimple center.} Let $\RR^{p,q}$ denote the real vector
space $\RR^{p+q}$ endowed with a metric $\la\,,\ra_{p,q}$ of signature $(p,q)$.
Let $\sso(p,q)$ denote the set of self adjoint transformations for
$\la\,,\ra_{p,q}$. This a semisimple Lie algebra and the Killing form $K$ a
natural ad-invariant metric on $\sso(p,q)$. Indeed $\sso(p,q)$ acts on 
$\RR^{p,q}$ just by evaluation. Take the direct sum as vector spaces 
$\nn=\sso(p,q)\oplus \RR^{p,q}$ and equipped with the product metric
$\la\,,\,\ra_{\nn}$ such that $\la\,,\,\ra_{\sso(p,q)\times \sso(p,q)}=K$, 
$\la\,,\,\ra_{\RR^{p,q}\times \RR^{p,q}}=\la\,,\,\ra_{p,q}$ and $\la\sso(p,q), \RR^{p,q}\ra=0$.
Thus  a Lie bracket can be defined on $\nn$ by
$$K( [u,v], A )=\la A u, v\ra_{p,q} \qquad \quad \mbox{ for all }u,v \in
\RR^{p,q}, A\in \sso(p,q).$$
The corresponding 2-step nilpotent Lie group equipped with the left invariant
metric induced  by the metric above, makes of $N$ a naturally reductive pseudo
Riemannian space (Theorem (\ref{t2}). 

A similar construction can be done by restriction of the evaluating action to  a non degenerate subalgebra of
$\sso(p,q)$. 

\vskip 2pt

\rv \, {\it Modified tangent semisimple.} The Killing 
form $K$ is an ad-invariant metric on any semisimple Lie algebra $\ggo$. 
As usual the tangent Lie algebra $\ct \ggo$ is the semidirect product $\ggo \ltimes
\ggo$ via the adjoint representation. 
We shall modify the algebraic structure on $\ct  \ggo$ in order to get a 
naturally reductive pseudo Riemannian 2-step nilpotent Lie group. 

Take the Lie algebra $\ggo$ together with the Killing form and let $\vv$
denote the underlying vector space to $\ggo$ endowed also with the Killing
form metric. To this pair $(\ggo, \vv)$ attach 

- the  metric given by $\la\,,\,\ra_{\ggo}=\la\,,\,\ra_{\vv}=K$ and
$\la\ggo,\vv\ra=0$; 

- the adjoint representation $\ad:\ggo \to \sso(\vv, K)$.

The adjoint representation is faithful and there 
is no trivial subrepresentations, so that $(\ggo, \vv, K+K)$ constitues a
data set for a 2-step nilpotent Lie group $N(\ggo, \vv)$ and by (\ref{t2})
it is naturally reductive pseudo Riemannian. Clearly the signature of this metric is twice
as much  the signature of $B$ and the isometry group can be computed with
(\ref{coro}).

Notice that whenever $\ggo$ is compact the procedure above is a case of
the construction for naturally reductive Riemannian nilmanifolds (see
(i)).

In the next section we shall see that the 2-step nilpotent Lie algebra above,
together with another metric gives rise to a Lie algebra carrying an
ad-invariant metric. 

\vskip 2pt

\section{Other examples of  naturally reductive metrics}

In this section we study 2-step nilpotent Lie algebras with ad-invariant metrics. The corresponding Lie group carries a bi-invariant metric for which the center is degenerate. 
  
 An {\em ad-invariant metric} on a Lie algebra $\ggo$ is a non degenerate symmetric
 bilinear map $\la\,,\,\ra:\ggo \times \ggo \to \RR$ such that 
\begin{equation}
\la [x,y], z\ra + \la y, [x,z]\ra =0 \qquad \mbox{ for all } x,y, z \in \nn.
\end{equation}
Recall that on a connected Lie group $G$ furnished with a left invariant pseudo
Riemannian metric $\la\,,\,\ra$, the following  statements are equivalent
(see \cite{ON} Ch. 11):

\begin{enumerate}
\item $\la\,,\,\ra$ is right invariant, hence bi-invariant;
\item $\la\,,\,\ra$ is $\Ad(G)$-invariant;
\item the inversion map $g\to g^{-1}$ is an isometry of $G$;
\item $\la [x,y], z\ra + \la y, [x,z]\ra =0$ for all $x,y, z \in \ggo$;
\item $\nabla_xy =\frac12 [x,y]$ for all $x,y\in \ggo$, where $\nabla$ denotes the
Levi Civita connection;
\item the geodesics of $G$ starting at $e$ are the one parameter subgroups of $G$.
\end{enumerate}

 Clearly $(G, \la\,,\,\ra)$ is naturally
reductive, which by (3)   is a symmetric space.  Furthermore one has

\begin{itemize}

\item the Levi-Civita connection is given by 
$$\nabla_x y=\frac 12 [x,y] \qquad \mbox{ for 
all } x,y \in \ggo,$$

\item the curvature tensor is 
$$R(x,y)=\frac14 \ad([x,y])\qquad \mbox{ for }x,y \in \ggo.$$
\end{itemize}

Hence any  {\em simply connected 2-step nilpotent Lie group equipped with a 
bi-invariant  metric is flat}.

The set of nilpotent Lie groups carrying a bi-invariant pseudo Riemannian metric is  non empty.  An element of this set is for instance the simply connected Lie group
whose Lie algebra is the free 3-step
nilpotent Lie algebra in two generators: in fact,
$\nn$ the Lie algebra spanned as vector space by $e_1, e_2, e_3, e_4, e_5$ with the non zero Lie brackets
$$[e_1, e_2]=e_3\qquad [e_1, e_3]=e_4\qquad [e_2, e_3]=e_5, $$
 carries the ad-invariant metric defined by the non vanishing symmetric relations 
$$\la e_3, e_3\ra=1=\la e_1, e_5\ra=-\la e_2, e_4\ra.$$
 
 Otherwise  in the Riemannian case, a naturally reductive  nilpotent Lie group may be at most 2-step nilpotent \cite{Go}.
 
 \vskip 3pt
 
 Let $\nn$ denote a 2-step nilpotent Lie algebra  with an
 ad-invariant metric $\la\,,\,\ra$. It is not hard to prove that $\zz^{\perp}=C(\nn)$ and
 therefore the center is always an isotropic ideal. Moreover $\nn$ decomposes as a
 orthogonal product 
 \begin{equation}\label{desad} \nn=\tilde{\zz} \times \tilde{\nn}
 \end{equation}
 where $\tilde{\zz}$ is a non degenerate central ideal and $\tilde{\nn}$ is a 2-step
 nilpotent ideal of corank zero, being  the corank of $\nn$ uniquely defined by the scalar $k:=\dim
 \zz -\dim C(\nn)$.  This follows essentially from the fact that the ad-invariant metric is non
 degenerate on any complementary subspace of $C(\nn)$ in $\zz$. Thus by choosing
 such a complement $\tilde{\zz}$, $\zz=\tilde{\zz}\oplus C(\nn)$ and 
 its orthogonal complement in $\nn$, $\nn=\tilde{\zz}\oplus \tilde{\zz}^{\perp}$ one gets a
 decomposition as above (\ref{desad}) with $\tilde{\nn}:=\tilde{\zz}^{\perp}$.
 
 Assume now the corank of $(\nn, \la\,,\,\ra)$ vanishes, so that
 $\zz^{\perp}=C(\nn)=\zz$. One can produce a isotropic subspace $\vv_1$ such that
 the ad-invariant metric on $\zz\oplus \vv_1$ is non degenerate. Hence one
 obtains a orthogonal decomposition as vector spaces 
 $$\nn=(\zz \oplus \vv_1)\oplus \vv_2, \qquad\mbox{ where }\vv_2=(\zz\oplus
 \vv_1)^{\perp}.$$
We claim $\vv_2=0$. In fact  for all $x\in \vv_2$ one has $\la x, [u,v]\ra=0$ for
all $u,v\in \nn$ 
implying that $x\in C(\nn)^{\perp}\cap \vv_2=\{0\}$. Hence there is a 
splitting of $\nn$ as a direct sum of the isotropic subspaces $\zz$ and $\vv$
so that the metric on $\nn$ is neutral:
$$\nn=\zz \oplus \vv.$$

Among other possible constructions, 2-step nilpotent Lie algebras admitting an
ad-invariant metric can be obtained as follows. Let $(\vv, \la \,,\, \ra_+)$ denote
a real vector space equipped with an inner product and let $\rho:\vv \to \sso(\vv)$
an injective linear map satisfying 
\begin{equation}\label{jad}
\rho(u)u=0\qquad\mbox{  for all }u\in \vv.
\end{equation}
 Consider the
vector space $\nn:=\vv^*\oplus \vv$ furnished with the canonical neutral metric
$\la\,,\,\ra$ and
define a Lie bracket on $\nn$ by
\begin{equation}\label{brad}
\begin{array}{rcl}
[x,y] & = & 0 \quad \mbox{ for  }x\in \vv^*, y\in \nn\quad \mbox{ and }\quad 
[\nn,\nn]\subseteq \vv^*\\
\la [u,v], w\ra & = & \la \rho(w) u,v\ra_+ \qquad \mbox{ for all } u,v,w\in \vv.
\end{array}
\end{equation}
 Then $\nn$ becomes a 2-step nilpotent Lie algebra of
corank zero for which the metric $\la\,,\,\ra$ is ad-invariant. This construction was called the
{\em modified cotangent}, since $\nn$ is linear isomorphic to the cotangent of $\vv$.
Notice that the commutator coincides with the center and it equals $\vv^*$. This
allows to construct 2-step nilpotent Lie algebras of null corank which carry an ad-invariant metric.  Furthermore this is basically
the way to obtain  such Lie algebras, see \cite{Ov}:

\begin{thm} \label{mod}  Let $(\nn, \la\,,\,\ra)$ denote a 2-step nilpotent Lie algebra of corank
$m$ endowed with an ad-invariant metric. Then $(\nn, \la\,,\,\ra)$ is isometric
isomorphic to an orthogonal direct product of the Lie algebras $\RR^m$ and a
modified cotangent.
\end{thm}

 One can get 2-step nilpotent examples by proceding as follows. Let $(\ggo, B)$ denote a compact
semisimple Lie algebra and $B$ its Killing form. Since  $B$ is negative
definite on $\ggo$, $-B$ determines an inner product
on $\ggo$. The adjoint map, $\ad:\ggo \to \sso(\ggo, B)$ satisfies
(\ref{jad}), therefore the vector space $\ggo^* \oplus \ggo$ equipped with the
Lie bracket defined in (\ref{brad}) makes of $\ggo^*\oplus \ggo$ a 2-step
nilpotent Lie algebra which carries an ad-invariant metric, the usual
neutral metric on $\ggo^*\oplus \ggo$.

From (\ref{jad}) it is clear that the non singular Lie algebras cannot carry an 
ad-invariant metric. Note that if a 2-step nilpotent Lie algebra admits and
ad-invariant metric, then $\dim \nn -\dim \zz=\dim C(\nn)$. This condition is
however not sufficient. 

A  self adjoint derivation $\phi$  in such a Lie algebra of zero corank has the form
$$\begin{array}{rcll}
\phi(z)& = &-A^*z\in \vv^*& \mbox{ for } z\in \vv^*\\
 \phi(v)& = & B v + A v &  \mbox{ where }Bv\in \vv^*, \,
Av\in \vv, \,\mbox{ for } v\in \vv
\end{array}
$$
and such that $A^*$ denotes the dual map of $A$: $A^*\varphi= \varphi \circ A$ for
$\varphi \in \vv^*$. On the other hand according to the results in \cite{Mu}  the isotropy group of isometries fixing the
identity element on the corresponding 2-step nilpotent Lie group consists of the self adjoint transformations with respect to $\la\,,\,\ra$. Thus
$$I_a(N)\subseteq I(N).$$

 Examples of 2-step nilpotent Lie algebras with ad-invariant metrics arise
  by taking $\ct^*\nn$, the cotangent of any 2-step
nilpotent Lie algebra $\nn$ together with the canonical neutral metric (see
(\ref{exad})). Let $\nn=\zz\oplus \vv$ denote a 2-step nilpotent Lie
algebra, where $\vv$ is any complementary subspace of $\zz$ in $\ggo$. Let $z_1, \hdots, z_m$ be a basis of the center $\zz$ and let $v_1,
\hdots, v_n$  be a  basis of the vector space $\vv$. Thus
$$[v_i, v_j]=\sum_{s=k}^m c_{ij}^s z_s \qquad \quad i, j=1, \hdots n.$$ 
Let $\ct^*\nn=\nn\ltimes \nn^*$ denote the cotangent Lie algebra obtained via
the coadjoint representation. Indeed  
 $z^1, \hdots, z^m, v^1,
\hdots, v^n$ becomes the  dual basis of the basis above adapted to the
decomposition $\nn^*=\zz^*\oplus \vv^*$. The non trivial Lie bracket relations concerning
the coadjoint action follow
$$
[v_i, z^j]=\sum_{s=1}^n d_{ij}^s v^s \qquad \quad \mbox{ for } i=1, \hdots n, j=1, \hdots m.
$$
Thus 
$[v_i, z^j](v_k)= d_{ij}^k$ and by the definition 
$$
[v_i, z^j](v_k) =-z^j(\sum_{s=1}^m c_{ik}^s z^s)= - c_{ik}^j \qquad \quad 
i,k=1, \hdots n, j=1, \hdots m.
$$
Therefore $d_{ij}^k= - c_{ik}^j$ for $i,k=1, \hdots n, j=1, \hdots m$.

It is clear that if for some basis of $\nn$ the structure constants are rational numbers then by choosing the union of this basis and its dual on $\ct^*\nn$ one gets
rational structure constants for $\ct^*\nn$. Thus by the Mal'cev criterium $N$
and  its cotangent $\ct^*N$, the simply connected Lie group with Lie algebra $\ct^*\nn$, admits
a lattice  which induces   a compact quotient  (see \cite{O-V,Ra} for instance).
 
Let $\Gamma\subset \ct^*N$ denote a cocompact lattice of $\ct^*N$. Indeed
$\ct^*N$ acts on the compact nilmanifold $(\ct^*N)/\Gamma$ by left translation isometries if we
induce to the quotient the bi-invariant metric corresponding to the neutral canonical one on
$\ct^*\nn$. The tangent
space at the representative $e$ can be identified with $\ct^*\nn \simeq T_e((\ct^*N)/\Gamma)$ so that
$\ct^*\nn=\{0\}\oplus \ct^*\nn$ and clearly $Ad(\Gamma)\ct^*\nn\subseteq
\ct^*\nn$ which says that $(\ct^*N)/\Gamma$ is homogeneous reductive. Moreover 
 the induced  metric on the quotient  satisfies
$$\la[x,y],z\ra+ \la [x,z], y\ra=0\qquad \quad \forall x,y, z\in \ct^*\nn.$$

\begin{prop} Let $N$ denote a 2-step nilpotent Lie group. If it admits a
cocompact lattice then  the cotangent Lie group $\ct^*N$ admits a cocompact
lattice $\Gamma$ such that $(\ct^*N)/\Gamma$ is pseudo Riemannian naturally reductive.
\end{prop}

\begin{exa} The low dimensional 2-step nilpotent Lie group $N$ admitting an 
ad-invariant metric occurs in dimension six. 
 This Lie algebra can be also be described as the cotangent of the Heisenberg Lie algebra $\ct^*\hh_3$. Explicitly let $e_1,e_2,e_3, e_4, e_5, e_6$ be a basis of $\nn$; the Lie brackets are
$$[e_4,e_5]=e_1 \qquad [e_4,e_6]=e_2 \qquad [e_5,e_6]=e_3$$
and the ad-invariant metric is defined by the non zero symmetric relations
$$1 = \la e_1, e_6\ra =\la e_2, e_5\ra =\la e_3, e_4\ra.$$

The corresponding simply connected six dimensional Lie group $N$ can be modelled on $\RR^6$ together with  the multiplication group  given by
$$\begin{array}{rcl}
(x_1,x_2,x_3, x_4,x_5,x_6) \cdot (y_1, y_2,y_3,y_4,y_5,y_6) &  = & (x_1 + y_1 + \frac12(x_4y_5-x_5y_4), \\
&& x_2+y_2+\frac12(x_4y_6-x_6y_4),\\
& &  x_3+y_3+\frac12(x_5y_6-x_6y_5),\\
&& x_4+y_4, x_5+y_5, x_6+y_6).
\end{array}
$$
By the Malcev criterium $N$ admits a cocompact lattice $\Gamma$. 
 By inducing the bi-invariant metric of $N$ to $N/\Gamma$ one gets a 
 invariant metric on  $N/\Gamma$, and in this way $N/\Gamma$ is
 a pseudo Riemannian naturally reductive compact nilmanifold.

For instance the subgroup of $N$ given by
$$\Gamma=\{(k_1,k_2,k_3, 2k_4, k_5, 2k_6)\,/\mbox{ for }\,k_i\in \ZZ \, \forall i=1,2,3,4,5,6\}$$
is a co-compact lattice of $N$, so that $N/\Gamma$ is a 
compact homogeneous manifold.

\end{exa}

\

{\sc Acknoledgements.} The author expresses her deep gratitud to Victor Bangert, for his support during the stay of the author at the Universit\"at Freiburg, where this work was done.

The author is partially supported by DAAD, CONICET, ANPCyT and  Secyt - Universidad Nacional de C\'ordoba.

\end{document}